\documentclass[12pt,a4paper,reqno]{amsart}
\usepackage[english]{babel}
\usepackage[applemac]{inputenc}
\usepackage[T1]{fontenc}
\usepackage{palatino}
\usepackage{amsmath}
\usepackage{amssymb}
\usepackage{amsthm}
\usepackage{amsfonts}
\usepackage{graphicx}
\usepackage{yhmath}
\usepackage[colorlinks = true, citecolor = black]{hyperref}
\title{Non-existence of multi-line Besicovitch sets}
\author{Tuomas Orponen}\email{tuomas.orponen@helsinki.fi}
\thanks{The author gratefully acknowledges the financial support of the Finnish National Doctoral Programme in Mathematics and its Applications.}
\subjclass[2010]{28A80 (Primary); 28A78, 42B25 (Secondary).}
\keywords{Besicovitch sets, Kakeya maximal operator}

\newcommand{\R}{\mathbb{R}}
\newcommand{\N}{\mathbb{N}}

\newcommand{\Z}{\mathbb{Z}}

\newcommand{\calT}{\mathcal{T}}
\newcommand{\calM}{\mathcal{M}}
\newcommand{\calR}{\mathcal{R}}

\newcommand{\calH}{\mathcal{H}}

\newcommand{\calL}{\mathcal{L}}

\newcommand{\spt}{\operatorname{spt}}

\newcommand{\diam}{\operatorname{diam}}
\newcommand{\m}{\mathfrak{m}}

\numberwithin{equation}{section}

\theoremstyle{plain}
\newtheorem{thm}[equation]{Theorem}
\newtheorem{lemma}[equation]{Lemma}

\newtheorem{proposition}[equation]{Proposition}

\theoremstyle{definition}
\newtheorem{definition}[equation]{Definition}

\theoremstyle{remark}
\newtheorem{remark}[equation]{Remark}

\begin{document}

\begin{abstract} If a compact set $K \subset \R^{2}$ contains a positive-dimensional family of line-segments in every direction, then $K$ has positive measure.
\end{abstract}

\maketitle

\section{Introduction} There are compact sets in the plane, which have zero Lebesgue measure, yet contain a line segment in every direction. Such a set was first constructed by A. S. Besicovitch in 1919, and this existence result is now one of the most widely known theorems in geometric measure theory, not least due to its profound consequences for Euclidean harmonic analysis. For a historical account of the problem and (some of) its connections, see \cite[\S 7]{Fa}.  In the present paper, we ask: what if a compact set contains \textbf{many} lines in every direction -- or even many directions? If the word 'many' is interpreted as in Theorems \ref{main} and \ref{main2}, the conclusion is that the set has to have positive Lebesgue measure. In other words, there exist no 'multi-line' Besicovitch sets. Our proof uses methods in harmonic analysis. More precisely, we extend Cordoba's proof \cite{Co} for the 'almost boundedness' of the Kakeya maximal operator from 1977. 

Given a direction $e \in S^{1}$ and a number $s \in [0,1]$, a family of line-segments $\calL$ perpendicular to $e$ is called \emph{s-dimensional}, if the union $L = \cup \calL \subset \R^{2}$ satisfies
\begin{displaymath} \calH^{s}(\rho_{e}(L)) > 0, \end{displaymath}
where $\rho_{e}$ stands for the orthogonal projection $\rho_{e}(x) = x \cdot e$, and $\calH^{s}$ is the $s$-dimensional Hausdorff measure, see \cite[Chapter 4]{Ma}. In case $s > 0$, the collection $\calL$ is \emph{positive-dimensional}. The definition imposes no conditions of measurability on $L$ or $\rho_{e}(L)$, even though we will actually need to know that the projections $\rho_{e}(L)$ are regular enough for Frostman's lemma to be applied. Fortunately, this is automatically satisfied, see Lemma \ref{measurability}. Our main result is:
\begin{thm}\label{main} Let $K \subset \R^{2}$ be a compact set containing the unions of positive-dimensional families of line-segments in $\calH^{1}$-positively many directions. Then $\calL^{2}(K) > 0$.
\end{thm}

It is not assumed that the set with '$\calH^{1}$-positively many directions' is measurable. Theorem \ref{main} is a corollary of the following slightly sharper result:
\begin{thm}\label{main2} Let $0 < s \leq 1$, and let $K \subset \R^{2}$ be a compact set containing the unions of $s$-dimensional families of line-segments in a set of directions $E_{0} \subset S^{1}$ with $\dim E_{0} > 1 - s$. Then $\calL^{2}(K) > 0$.
\end{thm}
Again, we require no regularity from $E$. A word on notation before we begin: we write $A \lesssim B$, if there exists a finite constant $C > 0$ such that $A \leq CB$. If $C$ is allowed to depend on a parameter, say, $p$, we may write $A \lesssim_{p} B$. The sequence $A \lesssim B \lesssim A$ is abbreviated to $A \asymp B$.

\section{Acknowledgements}

I am grateful to an anonymous referee for a careful reading of the manuscript and several useful comments.

\section{The multi-line maximal operator} Adopting Cordoba's approach in proving that ordinary Kakeya sets have dimension two, we first need to introduce a maximal operator suitable for our purposes. Our operator, defined on $S^{1}$ rather than $\R^{2}$, is similar to the modification of Cordoba's operator introduced by J. Bourgain in \cite{Bo}. If $e \in S^{1}$ and $\delta > 0$, we denote by $\calT_{e}^{\delta}$ the collection of disjoint $\delta$-tubes perpendicular to the line spanned by $e$. More precisely, if $\rho_{e} \colon \R^{2} \to \R$ is the orthogonal projection $\rho_{e}(x) = x \cdot e$, we set
\begin{displaymath} \calT_{e}^{\delta} := \{\rho_{e}^{-1}[j\delta,(j + 1)\delta) : j \in \Z\}. \end{displaymath}
Next, we introduce the family of $(\delta,e)$-rectangles, denoted by $\calR_{e}^{\delta}$. A rectangle $R \subset \R^{2}$ is a member of $\calR_{e}^{\delta}$, if $R$ is a $\delta \times 1$-rectangle, and $R \subset T$ for some $T \in \calT_{e}^{\delta}$. A set $B \subset \R^{2}$ is called a $(\delta,e)$-set, if $B \cap T \in \calR_{e}^{\delta}$ for every tube $T \in \calT_{e}^{\delta}$. Given $0 < s \leq 1$, a measure $\mu$ on $\R^{2}$ is called a $(\delta,e,s)$-measure, if $\mu$ is actually a function with the following properties:
\begin{itemize}
\item[(i)] there exists a $(\delta,e)$-set $B \subset \R^{2}$ such that
\begin{displaymath} \mu = \sum_{T \in \calT^{\delta}_{e}} a_{T} \cdot \chi_{B \cap T}, \end{displaymath}
\item[(ii)] the $L^{1}$-norm of $\mu$ is bounded by one,
\begin{displaymath} \delta \sum_{T \in \calT_{e}^{\delta}} a_{T} =  \|\mu\|_{L^{1}(\R^{2})} \leq 1, \end{displaymath}
\item[(iii)] the projection $\mu_{e} := \rho_{e\sharp}\mu$ satisfies the growth condition
\begin{displaymath} \mu_{e}(I) \leq \ell(I)^{s} \end{displaymath}
for every interval $I \subset \R$.
\end{itemize}
The parameter $s > 0$ will be thought as fixed, and the collection of all $(\delta,e,s)$-measures is simply denoted by $\calM_{e}^{\delta}$. 
\begin{definition}[Multi-line maximal operator] If $f \colon \R^{2} \to \R$ is a bounded Borel function, we set
\begin{displaymath} M^{\delta}f(e) := \sup \left\{ \int f \, d\mu : \mu \in \calM_{e}^{\delta} \right\}. \end{displaymath}
\end{definition}
\begin{remark} For a fixed bounded Borel function $f \colon \R^{2} \to \R$, the mapping $e \mapsto M^{\delta}f(e)$ is lower semicontinuous. Indeed, if $\mu^{e} \in \calM_{e}^{\delta}$, we may rotate $\mu^{e}$ to obtain measures $\mu^{\xi} \in \calM^{\delta}_{\xi}$, for $\xi \neq e$. As $\xi \to e$, the, difference $\int f \, d\mu^{e} - \int f \, d\mu^{\xi}$ tends to zero. As a consequence of the semicontinuity, the function $M^{\delta}f$ can be discretized by choosing a finite collection of vectors $\{e_{1},\ldots,e_{q}\} \subset S^{1}$, corresponding measures $\mu^{e_{j}} \in \calM_{e_{j}}^{\delta}$ and some numbers $\delta_{j} < \delta$ in such a manner that $S^{1}$ is covered by the balls $B(e_{j},\delta_{j})$, and
\begin{displaymath} M^{\delta} f(e) \asymp \sum_{j = 1}^{q} \left[ \int f\, d\mu^{e_{j}}\chi_{B(e_{j},\delta_{j})}(e) \right]. \end{displaymath}
In the following proofs, all measurability issues can be resolved by replacing $M^{\delta}f$ with the discretized version.
\end{remark}

\section{A restricted weak-type (2,2) bound for $M^{\delta}$}

Fix $s \in (0,1)$ for the rest of the paper. The central component in the proof of Theorem \ref{main2} is the following estimate:
\begin{proposition} Let $\sigma$ be a Borel measure on $S^{1}$ satisfying the bound $\sigma(B(e,r)) \leq r^{1 - t}$, $e \in S^{1}$, $r > 0$, for some $t < s$. Then, the maximal operator $M^{\delta}$ satisfies the weak-type $(2,2)$-estimate
\begin{equation}\label{weak} \sigma(\{e \in S^{1} : M^{\delta}\chi_{B}(e) \geq \lambda\})^{1/2} \lesssim_{t} \frac{\calL^{2}(B)^{1/2}}{\lambda}, \qquad \lambda > 0, \end{equation}
for compact sets $B \subset \R^{2}$, where the implicit constants are independent of $\delta > 0$.
\end{proposition}

\begin{proof} Fix $\lambda > 0$ and write $E := \{e \in S^{1} : M^{\delta}_{s}\chi_{B}(e) \geq \lambda\}$. For each $e \in E$, choose a measure $\mu^{e} \in \calM^{\delta}_{e}$ with $\int_{B} d\mu^{e} \gtrsim \lambda$. Then, since the measures $\mu^{e}$ are functions, we may estimate as follows:
\begin{align} \lambda\sigma(E) & \lesssim \int_{E} \int_{B} d\mu^{e} \, d\sigma(e) = \int_{B} \int_{E} \mu^{e}(x) \, d\sigma(e) \, dx \notag\\
& \leq \calL^{2}(B)^{1/2} \left(\int \left(\int_{E} \mu^{e}(x) \, d\sigma(e) \right)^{2} \, dx \right)^{1/2}\notag \\
&\label{form2} = \calL^{2}(B)^{1/2}\left(\iint_{E \times E} \left[ \int \mu^{e}(x)\mu^{\xi}(x) \, dx \right] \, d\sigma(e) \, d\sigma(\xi) \right)^{1/2}. \end{align} 
So, it remains to bound the correlation
\begin{displaymath} \int \mu^{e}(x)\mu^{\xi}(x) \, dx. \end{displaymath}
We first write
\begin{displaymath} \int \mu^{e}(x)\mu^{\xi}(x) \, dx = \sum_{j \in \Z} a_{j}^{\xi}\int_{R_{j}^{\xi}}\mu^{e}(x) \, dx, \end{displaymath}
where $R_{j}^{\xi} \in \calR^{\delta}_{\xi}$ is some $(\delta,\xi)$-rectangle, on which $\mu^{\xi}$ takes the constant value $a_{j}^{\xi}$, see Figure \ref{fig1}.

\begin{figure}[ht!]
\begin{center}
\includegraphics[scale = 0.6]{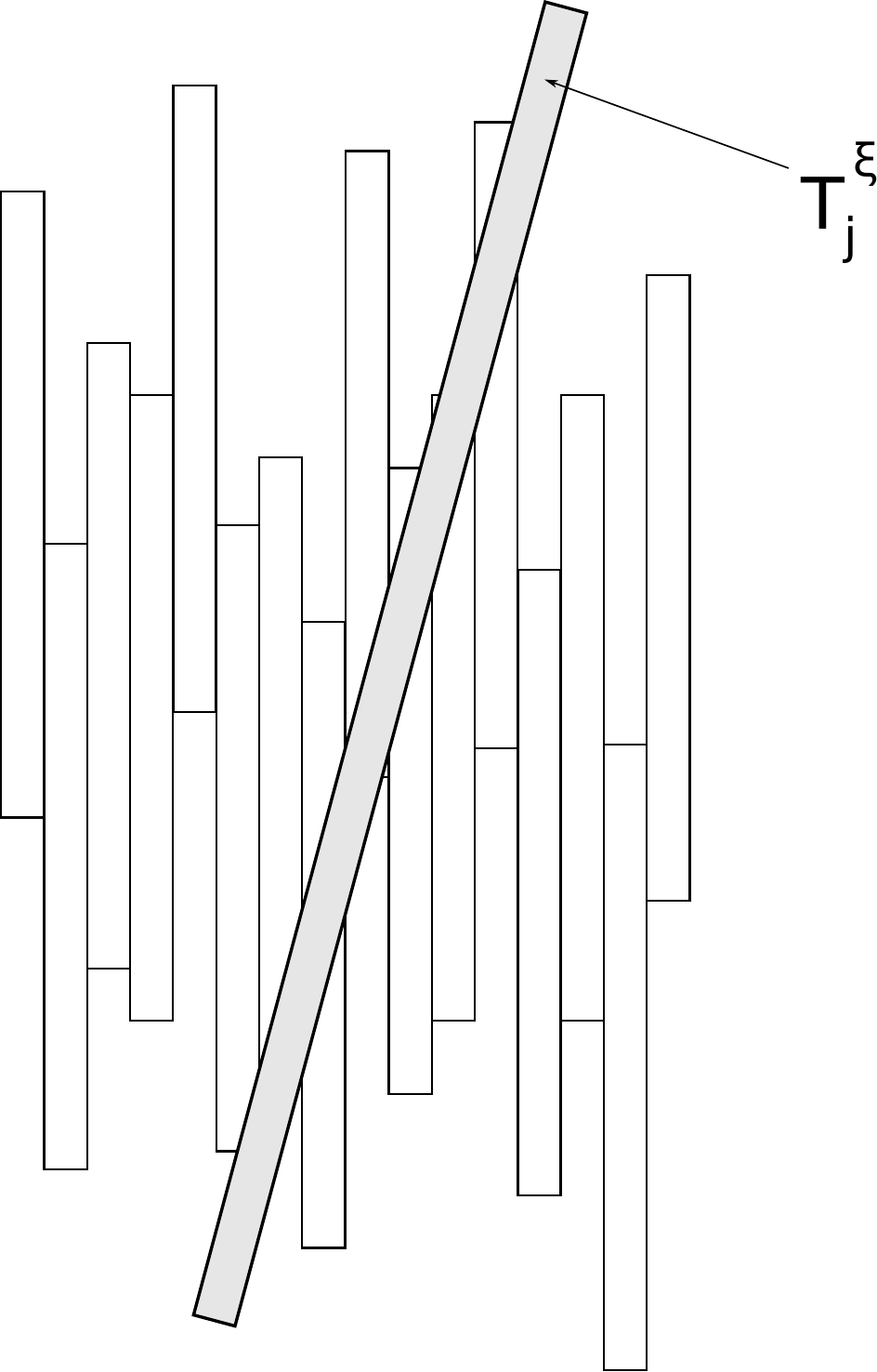}
\caption{The tube $T^{\xi}_{j} \supset R^{\xi}_{j}$ intersecting the rectangles $R^{e}_{i}$.}\label{fig1}
\end{center}
\end{figure}

Similarly expanding $\mu^{e}$ as a sum $\mu^{e} = \sum a_{i}^{e}\chi_{R_{i}^{e}}$, we have
\begin{displaymath} \int_{R_{j}^{\xi}} \mu^{e}(x) \, dx \leq \sum_{i : R_{i}^{e} \cap R_{j}^{\xi} \neq \emptyset} a_{i}^{e} \cdot \calL^{2}(R_{i}^{e} \cap R_{j}^{\xi}), \qquad j \in \Z. \end{displaymath}
The diameter of the intersection $R_{i}^{e} \cap R_{j}^{\xi}$ is bounded by 
\begin{displaymath} \diam(R_{i}^{e} \cap R_{j}^{\xi}) \lesssim \frac{\delta}{|e - \xi| + \delta}, \end{displaymath}
so $\calL^{2}(R_{i}^{e} \cap R_{j}^{\xi}) \lesssim \delta^{2}/(|e - \xi| + \delta)$. Finally, we have to estimate the sum of the numbers $a_{i}^{e}$ over the indices $\{i : R_{i}^{e} \cap R_{j}^{\xi} \neq \emptyset\}$. Using basic trigonometry, the projection of the rectangle $R_{j}^{\xi}$ onto the line spanned by $e$ is an interval $I_{e,\xi}$ of length 
\begin{displaymath} \ell(I_{e,\xi}) \lesssim |e - \xi| + \delta. \end{displaymath}
In particular, the rectangles $R_{i}^{e}$ with $R_{i}^{e} \cap R_{j}^{\xi} \neq \emptyset$ are all contained in the pre-image $\rho_{e}^{-1}(2I_{e,\xi})$. Combining this with assumption (iii),
\begin{align*} \sum_{i : R_{i}^{e} \cap R_{j}^{\xi}} a_{i}^{e} = \delta^{-1} \sum_{i : R_{i}^{e} \cap R_{j}^{\xi} \neq \emptyset} \int_{R_{i}^{e}} \mu^{e}(x) \, dx \leq \delta^{-1} (\mu^{e})_{e}(2I_{e,\xi}) \lesssim \frac{(|e - \xi| + \delta)^{s}}{\delta}. \end{align*} 
Putting everything together and using assumption (ii) yields
\begin{align*} \int \mu^{e}(x)\mu^{\xi}(x) \, dx & \lesssim \sum_{j \in \Z} \left( a_{j}^{\xi} \cdot \frac{\delta^{2}}{|e - \xi| + \delta} \cdot \frac{(|e - \xi| + \delta)^{s}}{\delta} \right)\\
& \leq |e - \xi|^{s - 1} \sum_{j \in \Z} \delta \cdot a_{j}^{\xi} \leq |e - \xi|^{s - 1}. \end{align*}
Inserting this back into \eqref{form2} leads to
\begin{displaymath} \lambda\sigma(E) \lesssim \calL^{2}(B)^{1/2} \left(\int_{E} \left[\int_{S^{1}} |e - \xi|^{s - 1} \, \sigma(\xi) \right] d\sigma(e) \right)^{1/2} \lesssim \calL^{2}(B)^{1/2} \sigma(E)^{1/2}. \end{displaymath}
This concludes the proof of \eqref{weak}. The growth bound assumed from $\sigma$ was used above to obtain
\begin{displaymath} \int_{S^{1}} |e -\xi|^{s - 1} \, d\sigma(e) \lesssim_{t} 1. \end{displaymath} 
The proof of this is standard issue, see for example \cite[p. 109]{Ma}. 
\end{proof}

The following lemma addresses the measurability issues related to the projections $\rho_{e}(L)$, mentioned at the beginning of the introduction.

\begin{lemma}\label{measurability} Let $K \subset \R^{2}$ be a compact set, let $e \in S^{1}$ and let $c > 0$. Let $\calL$ be the collection of all line-segments contained in $K$, which are perpendicular to $e$ and have length at least $c$. Then $L := \cup \calL$ is compact; in particular, $\rho_{e}(L)$ is compact.
\end{lemma}

\begin{proof} Of course, we only need to verify that $L$ is closed. Fix $x \in \overline{L}$. We first pick a sequence of points $(x_{i})_{i \in \N}$ in $L$ with $x_{i} \to x$, and note that each point $x_{i}$ is contained in some line-segment $\ell_{i} \in \calL$. We then use the Blaschke selection theorem \cite[Theorem 3.16]{Fa} to produce a subsequence $(\ell_{i_{j}})_{j \in \N}$, convergent in the Hausdorff metric to a compact set $\ell \subset K$. It is clear from the definition of convergence in the Hausdorff metric that $x \in \ell$, and $\ell$ is a subset of the line $\rho_{e}^{-1}\{t\}$, where $t := \rho_{e}(x)$. Moreover, $\ell$ is connected according to \cite[Theorem 3.18]{Fa}. Thus, $\ell$ is either a line-segment or a point contained in $K \cap \rho_{e}^{-1}\{t\}$. Finally, we observe that $\ell$ has length at least $c$, since the $\delta$-neighborhoods of $\ell$ contain some line-segments $\ell_{i}$ for any $\delta > 0$. We conclude that $\ell \in \calL$, and so $x \in \ell \subset L$.
\end{proof} 

Now we are equipped to prove Theorem \ref{main2}

\begin{proof}[Proof of Theorem \ref{main2}]  Let $K \subset \R^{2}$ be as in Theorem \ref{main2}. Let $\calL^{e}_{c}$ be the collection of all line-segments contained in $K$, perpendicular to $e$ and with length at least $c > 0$. Write $L^{e} := \cup \calL^{e}_{c}$. Choosing $c > 0$ and $\alpha > 0$ small enough, the set $E = \{e \in S^{1} : H^{s}(\rho_{e}(L^{e})) \geq \alpha\}$ has dimension $\dim E > s - 1$, where $H^{s}$ stands for $s$-dimensional Hausdorff content. Without loss of generality, we may assume that $c = 1$. Thus, $\calL^{e} := \calL^{e}_{1}$ consists of line-segments of length at least one, and $\rho_{e}(L^{e})$ is a compact set according to Lemma \ref{measurability}. We can now apply Frostman's lemma to the sets $\rho_{e}(L^{e})$: it follows from the standard proof of this lemma, see \cite[Theorem 8.8]{Ma}, that for every $e \in E$ we can locate a measure $\tilde{\nu}^{e}$, supported on $\rho^{e}(L^{e})$, such that $\tilde{\nu}^{e}(I) \leq \ell(I)^{s}$ for every interval $I \subset \R$, and $1 \lesssim \tilde{\nu}^{e}(\R) \leq 1$, where the implicit constants only depend on $\alpha$ and $s$. To produce from $\tilde{\nu}^{e}$ a measure in $\calM_{e}^{\delta}$, we first discretize $\tilde{\nu}^{e}$ by defining
\begin{displaymath} \nu^{e} = \frac{1}{10}\sum_{j \in \Z} \frac{\tilde{\nu}^{e}[j\delta,(j + 1)\delta)}{\delta}\chi_{[j\delta,(j + 1)\delta)}. \end{displaymath}
Then $\nu^{e}$ is a measure with total mass $\nu^{e}(\R) = \tilde{\nu}^{e}(\R)/10 \gtrsim 1$. The factor $1/10$ is there only to ensure that $\nu^{e}$ satisfies the growth condition $\nu^{e}(I) \leq \ell(I)^{s}$. Whenever $\nu^{e}[j\delta,(j + 1)\delta) = \tilde{\nu}^{e}[j\delta,(j + 1)\delta) > 0$, we know that the intersection $\spt \tilde{\nu}^{e} \cap [j\delta,(j + 1)\delta]$ is non-empty. Recalling the definition of $\tilde{\nu}^{e}$, this means that the intersection $\rho_{e}^{-1}[j\delta,(j + 1)\delta] \cap K$ contains an entire unit line-segment. It follows that we may find a rectangle $R^{j} \in \calR^{\delta}_{e}$ contained in the intersection $\rho_{e}^{-1}[j\delta,(j + 1)\delta) \cap K(\delta)$. For each $j$, we choose one -- and only one -- such rectangle $R^{j}$ and define
\begin{displaymath} \mu^{e} = \sum_{j \in \Z} \frac{\nu^{e}[j\delta,(j + 1)\delta)}{\delta}\chi_{R^{j}}. \end{displaymath}
Then the projection $(\mu^{e})_{e}$ coincides with $\nu^{e}$ and, consequently, satisfies the growth condition in assumption (iii); the assumptions (i) and (ii) are clearly satisfied as well. We conclude that $\mu^{e} \in \calM^{\delta}_{e}$, whence 
\begin{displaymath} M^{\delta}\chi_{K(\delta)}(e) \geq \int_{K(\delta)} \, d\mu^{e} = \mu^{e}(\R^{2}) \gtrsim 1. \end{displaymath}
This holds for every direction $e \in E$, so there exists a constant $\m > 0$ such that the sets
\begin{displaymath} E^{\delta} := \{e \in S^{1} : M^{\delta}\chi_{K(\delta)}(e) > \m\} \end{displaymath}
contain $E$, for every $\delta > 0$. In particular, we may find $t < s$ such that the numbers $\calH^{1 - t}(E^{\delta})$ have a uniform lower bound $\beta > 0$. The sets $E^{\delta}$ are open, so another application of (the proof of) Frostman's lemma yields measures $\sigma^{\delta}$, $\delta > 0$, supported on $E^{\delta}$ and satisfying the bounds $\sigma^{\delta}(E^{\delta}) \gtrsim 1$ and $\sigma^{\delta}(B(e,r)) \leq r^{1 - t}$, where the implicit constants are again independent of $\delta > 0$. It remains to apply the weak-type estimate \eqref{weak} as follows:
\begin{displaymath} 1 \lesssim \sigma^{\delta}(E^{\delta}) = \sigma^{\delta}(\{e \in S^{1} : M^{\delta}\chi_{K(\delta)}(e) > \m\}) \lesssim \frac{[\calL^{2}(K(\delta))]^{2}}{c^{2}}. \end{displaymath}
Letting $\delta \to 0$ shows that $\calL^{2}(K) > 0$ and completes the proof of Theorem \ref{main2}. \end{proof}


\begin{thebibliography}{9999}
\bibitem[Bo]{Bo} \textsc{J. Bourgain}: \emph{Besicovitch type maximal operators and applications to Fourier analysis}, Geom. Funct. Anal. \textbf{1}, No. 2 (1991), pp. 147--187
\bibitem[Co]{Co} \textsc{A. Cordoba}: \emph{The Kakeya maximal function and spherical summation multipliers}, Amer. J. Math. \textbf{99} (1977), pp. 1-22
\bibitem[Fa]{Fa} \textsc{K. Falconer}: \emph{The geometry of fractal sets}, Cambridge University Press, 1985
\bibitem[Ma]{Ma} \textsc{P. Mattila}: \emph{Geometry of Sets and Measures in Euclidean Spaces}, Cambridge University Press, 1995
\end{thebibliography}
\end{document}